\documentclass[11pt,reqno]{amsart}

\usepackage{amsfonts}
\usepackage{amssymb}
\usepackage{verbatim}
\newtheorem{theorem}{Theorem}[section]
\newtheorem{corollary}[theorem]{Corollary}

\theoremstyle{definition}
\newtheorem{definition}[theorem]{Definition}
\newtheorem{example}[theorem]{Example}

\newtheorem{remark}{\it{\textbf{Remark}}}
\newcommand{\R}{\mathbb R}
\newcommand{\al}{\alpha}\newcommand{\si}{\sigma}

\newcommand{\Gdet}{\mathrm{Gdet}}
\newcommand{\sign}{\mathrm{sign}}
\def\mat#1#2{\mbox{$\left(\begin{array}{#1}#2\end{array}\right)$}}
\numberwithin{equation}{section}

\begin{document}	
	\title{A Generalized Determinant of Matrices and Applications}
	\date{}

\author[X.S. Lu]{Xuesong Lu}

\author[S.T. Mao]{Songtao Mao}

\author[Z.X. Wang]{Zixing Wang}

\author[Y.H. Zhang]{Yuehui Zhang$^*$}
\address{School of Mathematical Sciences, Shanghai Jiao Tong University,
800 Dongchuan Road, 200240 Shanghai, China}
\email{leocedar@qq.com, \  jarlly678@gmail.com,\  nbwzx@126.com, \ zyh@sjtu.edu.cn}

\thanks{Supported by NSFC under Grant No. 11771280 and 11671258, by NSF of Shanghai Municipal under Grant No. 17ZR1415400\\
\indent *\ Corresponding author}


\date{}

\dedicatory{}


\begin{abstract}

    A generalized definition of the determinant of matrices is given, which is compatible with the usual determinant for square matrices and keeps
    many important properties, such as being an alternating multilinear function, keeping multiplication formula and partly keeping the
    Cauchy-Binet’s formula. As applications of the new theory,  the generalized Cramer's rule and the generalized oriented volume are obtained. 

\end{abstract}

\maketitle

    \section{Introduction}

    The determinant is such a useful tool containing many nice properties but a defect that it is only available to square matrices. The exploration of the definition for non-square matrices can date back to 1925 in C.E. Cullis' book (\cite{C}), where the new term `{\it determinoid}' was created for non-square matrices. Lacking geometric meanings, Cullis' definition was overlooked for more than 40 years until Radi\'{c} found a new one (\cite{R1}) in 1966. Unfortunately, Radi\'{c}'s new definition was again overlooked for about 15 years until 1980, when in Joshi's paper (\cite{O}) a reasonable definition was given. Aiming to study the Penrose inverse of rectangular matrices, Joshi's determinant kept the cofactor expansion formula of square matrices and henceforth kept the determinant criterion of invertible matrices, but it failed to keep many other important properties of determinants. In 2001, Radi\'{c} himself gave a geometric application (\cite{R2}) of his definition for $2\times n$ matrices and A.Makarewicz, P. Pikuta and D.Szalkowski extended it to higher dimensions (\cite{MPS}). Other work in this topic including Amiri, Fathy, Bayat's definition (\cite{AFB}) in 2010. But all the definitions above have a common defect: they don't contain a meaning of volume, which is a great meaning of the usual determinant. Yanai, Takane and Ishii's definition (\cite{YTI}) in 2006 almost did it, but for lack of a sign system, their definition is not compatible with the usual definition for square matrices and hence can't express the oriented volume. Hence a sign system is necessary. After summing up all the advantages and disadvantages, we finally find a fresh definition equipping with a sign system that is compatible with the usual definition as well as keeps many properties such as being an alternating multilinear function, keeping multiplication formula and partly keeping the Cauchy-Binet’s formula. Furthermore, based on the sign system, we obtain the generalized Cramer's rule and the generalized oriented volume of a parallel polyhedron.

    Throughout, all matrices involved are over the field $\R$ of real numbers. $\R^{m\times n}$ stands for the vector spaces over $\R$ consisting of all $m\times n$ real matrices, $I_n$ the $n \times n$ identity matrix, $e_i$ the $i$-th column of $I_m$ for all $1\le i\le m$. We denote the symmetric group on $\underline{m}=\{1, 2, \cdots, m\}$ by $S_m$, the identity element of $S_m$ is denoted by $(1)$.
    
\section{A sign system of the determinant of matrices}
As stated in the introduction, a sign system is necessary for a determinant of matrices, hence we define the sign as the start step.

First, we define an action of the symmetric group $S_{m}$ on the vector space $\mathbb{R}^{m\times n}(m\ge n)$. It is well known that there is an isomorphism between symmetric group $S_m$ and all $m\times m$ permutation matrices $P_m$ as follows. For any $\sigma\in S_m$,
\begin{equation}
\begin{aligned}
    \varPhi:  S_m &\rightarrow  P_m\\ 
    \sigma &\mapsto \mat{c}{e^T_{\si(1)}\\ e^T_{\si(2)}\\ \vdots\\ e^T_{\si(m)}}.
    \end{aligned}
\end{equation}

For any $\sigma\in S_m, A\in\R^{m\times n}$, define
\begin{equation}\label{action}
\si(A):= \varPhi(\sigma) A=\mat{c}{e^T_{\si(1)}\\ e^T_{\si(2)}\\ \vdots\\ e^T_{\si(m)}}A=\mat{c}{e^T_{\si(1)}A\\ e^T_{\si(2)}A\\ \vdots\\ e^T_{\si(m)}A}.
\end{equation}
This is clearly a group action. In particular, we denote by $P_\sigma$ the action of $\sigma$ on the $m\times n(m\ge n)$ `Identity matrix' $I_{m\times n}=(e_1, e_2, \cdots, e_n)$, namely
\begin{equation}\label{actionI}
P_\sigma=\si(I)=\varPhi(\sigma) I_{m\times n}=\mat{c}{e^T_{\si(1)}\\ e^T_{\si(2)}\\ \vdots\\ e^T_{\si(m)}} .
\end{equation}

We associate for each $m\times n(m\ge n)$ matrix $A$ and each $\si\in S_m$ with the following square matrix in $\R^{n\times n}$ defined by
\begin{equation}\label{associate}
A^\si=P_\sigma^T A.
\end{equation}

 Note that if $A$ is of full column rank, there exists $\si\in S_m$ such that $A^\si$ is non-singular.  We now use the formula (\ref{associate})  to define the sign of our determinant for $m\times n(m\ge n)$ matrices. To this end, we at first define a total order on the symmetric group $S_{m}$. $\forall \si\not= \tau\in S_m$, define
\begin{equation}\label{order}
   \si>\tau \iff \exists\ t\in\underline{m}\ such\ that\ \si(t)>\tau(t), \si(j)=\tau(j), \forall j>t. 
\end{equation}

The largest element of $S_m$ under this order is the identity $(1)$ and the least one is $\si: i\mapsto m-i+1$. For each $A\in\mathbb{R}^{m\times n}$,  define
\begin{equation}\label{max}
\sigma_A=\begin{cases}
\ \ (1),\quad  \text{ if $A$ is not of full column rank},\\
\max\{\sigma\ |\ \det(A^\si)\ne 0,\ \sigma\in S_m\},\quad \text{otherwise}.\end{cases} 
\end{equation}

The key of the sign to be defined is the concept {\it principal submatrix} of a matrix $A\in\R^{m\times n}$, by which we mean the following square matrix:
\begin{equation}\label{principal}
A_P=A^{\si_A}=P_{\si_A}^TA\in\R^{n\times n}.
\end{equation}

\begin{definition}\label{2d1-sign}   Let $m\ge n$. Let $A\in\mathbb{R}^{m\times n}$. The sign of $A$, $\sign(A)$, is defined to be the sign of the determinant of its principal submatrix, that is
 $$\sign(A)=
    \begin{cases}
 \  \ 1,& \text{if}\ \ \det(A_P)>0,\\
 \ \ 0, & \text{if}\ \ \det(A_P)=0,\\
    -1,& \text{if}\ \ \det(A_P)<0.    \end{cases}
  $$

\end{definition}

\begin{remark}
The sign of $A$ is actually the same as that of the first $n\times n$ submatrix whose determinant is nonzero, where the ordering of $n\times n$ submatrices $A_{i_{1},i_{2},\ldots, i_{n}}$ is in lexicographic order of $\left(i_{1},i_{2},\ldots, i_{n}\right)$. In particular, if $\det(A_{1,2,\ldots,n})\ne0$, then $\sign(A)=\sign(A_{1,2,\ldots,n})$.
\end{remark}

\begin{remark}
If $m<n$, $A\in \R^{m\times n}$, we define $\sign(A)\equiv 0$. The rationality will be clarified in the next section.

Based on this sign system, we give a fresh definition of the determinant for matrices in the following.
\end{remark}

\begin{theorem}\label{signlemma} 
Let $m\ge n, A\in\R^{m\times n}$. Then
\begin{itemize} 
    \item[(1)] If $m=n$, then $\sign(A)=\sign(\det(A)) .$
    \item[(2)] If $A$ is a rectangular identity matrix, then $\sign(A)=1$.
    \item[(3)] If $B\in \R^{n\times n}$, then $\sign(A B)=\sign(A) \sign(B) .$
    \item[(4)] Let $A=(\al_1,\cdots, \al_n)$ , $B=(\al_1, \cdots,\al_{i-1},\al_i+k \al_j,\al_{i+1},\cdots,\al_n)$, then   $\sign(B)=\sign(A)$.
    \item[(5)] Let $k\in \R$, $A=(\al_1,\cdots, \al_n)$, $B=(\al_1,\cdots,k \al_i,\cdots,\al_n)$, then $\sign (B)=\sign(k)\sign(A)$.
  
\end{itemize}
\begin{proof}
(1)The matrix $A_{i_{1}\ldots i_{n}}$ in the right hand side of Definition \ref{Gdet} is indeed $A$ itself, so $\sign(A)=\sign(\det(A)) .$

(2)The matrix $A_{i_{1}\ldots i_{n}}$ in the right hand side of Definition \ref{Gdet} is an $n \times n$ identity matrix , so $\sign(A)=\sign(\det(A)) .$
       
(3)If $A$ or $B$ is not of full column rank, then $AB$  is not of full column rank. Hence the determinant of every submatrix of $AB$ is 0, i.e.  $\sign(AB)=0=\sign(A)\sign(B)$. 
Thus assume that both $A$ and $B$ are of full column rank. Since $B$ is invertible, $\sign(B)\ne0$. By the formula (\ref{principal}) and Definition \ref{2d1-sign}, the sign of $A$ is the same as that of its principle submatrix $A_P$: $\sign(A)=\sign(A_P)$. Note that the principal submatrix of $AB$ is $A_PB$, since $B$ is invertible, so $\sign(AB)=\sign(A_PB)$. It is well-known that $\sign(XY)=\sign(X)\sign(Y)$ for square matrices $X$ and $Y$, hence $\sign(A_PB)=\sign(A_P)\cdot \sign (B)$. Therefore, $\sign(AB)=\sign(A)\sign(B)$, as required.

(4)According to the basic property of the determinant of square matrices, $\forall 1\leq i_{1}<\ldots<i_{n}\leq m$, $\det(B_{i_{1},\ldots, i_{n}})=\det(A_{i_{1},\ldots, i_{n}})$. Thus $\sign(B)=\sign(A)$.

(5)According to the basic property of the determinant of square matrices, $\forall 1\leq i_{1}<\ldots<i_{n}\leq m$, $\det(B_{i_{1},\ldots, i_{n}})=k\det(A_{i_{1},\ldots, i_{n}})$. Thus $\sign (B)=\sign(k)\sign(A)$.

\end{proof}
\end{theorem}

\begin{corollary}\label{QR} 
Let $A\in \R^{m\times n}$ be of full column rank, $A=QR$ the $QR$ decomposition of $A$, where $Q\in \R^{m\times n}$, $Q^TQ=I_{n\times n}$, $R\in \R^{n\times n}$ is an upper triangular matrix with positive diagonal elements. Then $\sign(A)=\sign(Q)$.
\begin{proof}
By Theorem \ref{signlemma}(3), $\sign(A)=\sign(Q)\sign(R)$. Since the diagonal elements of $R$ are all positive numbers,  $\sign(R)=1$. Therefore, $\sign(A)=\sign(Q)$.
\end{proof}
\end{corollary}

\section{The Definition and Properties of $\Gdet$}

\begin{definition}\label{Gdet}  
Let $A\in\mathbb{R}^{m\times n}$. The generalized determinant of $A$, $\Gdet(A)$, is defined as
    $$    \Gdet(A):=\begin{cases}
       \sign\left(A\right)\left(\sum\limits_{1\leq i_{1}<\ldots<i_{n}\leq m}\left(\det(A_{i_{1},\ldots, i_{n}})\right)^{2}\right)^{\frac 12},& m\geq n,\\ 
       0,& m<n,\\  
       \end{cases}
    $$
where $A_{i_{1},\ldots, i_{n}}$ is the  $n\times n$ matrix $\mat{c}{A^Te_{i_{1}},  \cdots, A^Te_{i_{n}}}^T$.

\end{definition}

\begin{remark}\label{remark3}
There seems to be another natural way to define the $\Gdet$ when $m<n$, that is $\Gdet(A):=\Gdet(A^T)$. But if so, many good properties and applications will be lost such as the multiplication formula (Theorem \ref{multiplication}), the generalized Cramer's rule (Theorem \ref{Cramer}) and so on. Even Theorem \ref{coincide} below will fail. 
\end{remark}

  The definition $\ref{Gdet}$ is compatible with the usual one as shown in the next theorem.
  
  \begin{theorem}\label{square}
  Let $A\in \R^{n\times n}$. Then $\Gdet(A)=\det(A)$.
  \end{theorem}
  
  \begin{proof}
  For $A\in\R^{n\times n}$, the matrix $A_{i_{1}\ldots i_{n}}$ in the right hand side of Definition \ref{Gdet} is indeed $A$ itself, hence
$$
|\Gdet(A)|=
       \left(\sum\limits_{1\leq i_{1}<\ldots<i_{n}\leq n}\left(\det(A_{i_{1},\ldots, i_{n}})\right)^{2}\right)^{\frac 12}=|\det(A)|.
$$
          
By Theorem \ref{signlemma} (1), $\sign(A)=\sign(det(A))$.

Hence  $\Gdet(A)=\det(A)$
  \end{proof}

  To calculate a $\Gdet$ through the definition, one should calculate the determinants $\dbinom{m}{n}$ times. The following theorem cuts the process into only once. 
  \begin{theorem}\label{coincide}
  Let $A\in\R^{m\times n}$. Then $|\Gdet (A)|=(\det(A^T A))^{\frac 12}$.
    
\end{theorem}

\begin{proof}
If $m<n$, then $|\Gdet (A)|=0=(\det(A^T A))^{\frac 12}$. 

If $m\ge n$, by Cauchy-Binet's formula,
$$\det(A^TA)=\sum\limits_{1\leq i_{1}<\ldots<i_{n}\leq m}\left(\det(A_{i_{1},\ldots, i_{n}})\right)^{2}.$$
Then $$\left(\det(A^TA)\right)^{\frac 12}=\left(\sum\limits_{1\leq i_{1}<\ldots<i_{n}\leq m}\left(\det(A_{i_{1},\ldots, i_{n}})\right)^{2}\right)^{\frac 12}=|\Gdet (A)|.$$
.
\end{proof}
  In the reference \cite{YTI}, they use the formula in Theorem \ref{coincide} as their definition of nonsquare matrices. 

  If the definition of the case $m<n$ is $\Gdet(A):=\Gdet(A^T)$ as in Remark \ref{remark3}, then Theorem \ref{coincide} will fail. For example, let $A=(1,0)$, then
  $$
   |\Gdet (A)|=1\ne 0=(\det(A^T A))^{\frac 12}.
  $$


From Theorem \ref{coincide}, the generalized determinant of a matrix is closely related to its singular values.

\begin{corollary}
Let $m\ge n, A\in\R^{m\times n}$. Let $s_i \ (1 \le i \le n)$ be the singular values of $A$. Then $\left|\Gdet(A)\right|=\displaystyle\prod_{i=1}^ns_i$.
\end{corollary}

   Besides, $\Gdet$ also keeps many other important properties of the usual determinant. 

\begin{theorem}\label{multialf} $\Gdet$ is the alternating multilinear function over $(\R^m)^n$ taking value $1$ at $(e_1, \cdots, e_n)\in(\R^m)^n$. In other words,  $\Gdet$ is the function from  $(\R^m)^n$ to $\R$ satisfying the following conditions ( $k\in\R, \al_i\in\R^m, 1\le i\le n$):

\noindent$(1)$ $\Gdet (\al_1,\cdots, \al_n)=\Gdet (\al_1, \cdots,a_{i-1},a_i+k a_j,a_{i+1},\cdots,\al_n).$

\noindent$(2)$ $\Gdet (\al_1, \cdots,ka_{i},\cdots,\al_n)=k\Gdet (\al_1,\cdots, \al_n).$

\noindent$(3)$ $\Gdet (\al_1, \cdots, \al_{i}, \al_{i+1}, \cdots, \al_n)=-\Gdet (\al_1, \cdots, \al_{i+1}, \al_{i}, \cdots, \al_n).$

\noindent$(4)$ $\Gdet(e_1, \cdots, e_n)=1$.

    \end{theorem}
\begin{proof}
$(1)$ Denote $(\al_1,\cdots, \al_n)$ by $A$, $(\al_1, \cdots,\al_{i-1},\al_i+k \al_j,\al_{i+1},\cdots,\al_n)$ by $B$. From the Definition \ref{Gdet} 
$$
\begin{aligned}
|\Gdet (B)|&= \left(\sum\limits_{1\leq i_{1}<\ldots<i_{n}\leq m}(\det(B_{i_{1},\ldots, i_{n}}))^{2}\right)^{\frac 12}\\
&=\left(\sum\limits_{1\leq i_{1}<\ldots<i_{n}\leq m}\left(\det(A_{i_{1},\ldots, i_{n}})\right)^{2}\right)^{\frac 12}\\
&=|\Gdet (A)|.
\end{aligned}
$$ 

By Theorem \ref{signlemma} (4), $\sign(B)=\sign(A)$.

So $\Gdet (B)=\Gdet (A).$

$(2)$ Again denote $(\al_1,\cdots, \al_n)$ by $A$, denote $(\al_1,\cdots,k \al_i,\cdots,\al_n)$ by $B$. From the Definition \ref{Gdet} 
$$
\begin{aligned}
|\Gdet (B)|&= \left(\sum\limits_{1\leq i_{1}<\ldots<i_{n}\leq m}(\det(B_{i_{1},\ldots, i_{n}}))^{2}\right)^{\frac 12}\\
&=\left(\sum\limits_{1\leq i_{1}<\ldots<i_{n}\leq m} k^2\left(\det(A_{i_{1},\ldots, i_{n}})\right)^{2}\right)^{\frac 12}\\
&=\left|k\right|\left(\sum\limits_{1\leq i_{1}<\ldots<i_{n}\leq m} \left(\det(A_{i_{1},\ldots, i_{n}})\right)^{2}\right)^{\frac 12}\\
&=|k\Gdet (A)|.
\end{aligned}
$$

By Theorem \ref{signlemma} (5), $\sign (B)=\sign(k)\sign(A)$

So $\Gdet (B)=k\Gdet (A).$

$(3)$ By $(1)$ and $(2)$, 
$$
\begin{aligned}
&Gdet (\al_1, \cdots, \al_{i}, \al_{i+1}, \cdots, \al_n)\\
=&\Gdet (\al_1, \cdots, \al_{i}+\al_{i+1}, \al_{i+1}, \cdots, \al_n)\\
=&\Gdet (\al_1, \cdots, \al_{i}+\al_{i+1}, -\al_{i}, \cdots, \al_n)\\
=&\Gdet (\al_1, \cdots, \al_{i+1}, -\al_{i}, \cdots, \al_n)\\
=&-\Gdet (\al_1, \cdots, \al_{i+1}, \al_{i}, \cdots, \al_n).
\end{aligned}
$$

$(4)$ Let $A=(e_1, \cdots, e_n)$, then
$$
	\left|A_{i_{1},\ldots, i_{n}}\right|=\begin{cases}
	1,&		\text{if\,\,}(i_{1},\cdots,i_{n})=(1,2,\cdots,n)\,\, \text{for\,\,all\,\,}  1\le i\le n,\\
	0,&		\text{else},\\
	\end{cases}
$$
where $A_{i_{1},\ldots, i_{n}}$ is the  $n\times n$ matrix $\mat{c}{A^Te_{i_{1}},  \cdots, A^Te_{i_{n}}}^T$.

Hence
$$    \Gdet(A)=
       \sign\left(A\right)\left(\sum\limits_{1\leq i_{1}<\ldots<i_{n}\leq m}\left(\det(A_{i_{1},\ldots, i_{n}})\right)^{2}\right)^{\frac 12}=1.    $$
\end{proof}

%
%
%

An immediate consequence of Theorem \ref{multialf} is the following corollary.

    \begin{corollary}\label{leftinverse} Let  $A\in\R^{m\times n}$. The followings are equivalent:

    $(1)$ $\Gdet(A)\not=0$.

    $(2)$ $A$ is of full column rank.

    $(3)$ $A$ has a left inverse, that is,  $\exists B\in\R^{n\times m}$ such that $BA = I_n$.
\end{corollary}

 \begin{proof} The equivalence of (1) and (2) is clear from Theorem \ref{multialf}, and the equivalence of (2) and (3) is an easy exercise of Linear Algebra.
\end{proof}

Another remarkable advantage of Definition \ref{Gdet} is that it does satisfy the multiplication formula, as stated in the following

    \begin{theorem}\label{multiplication} Let  $A\in\R^{m\times n}, B\in\R^{n\times n}$. Then $\Gdet(AB)=\Gdet(A)\Gdet(B)$.

\end{theorem}
    \begin{proof}  If $m<n$, then $\Gdet(AB)=0=\Gdet(A)\Gdet(B)$. 
    
    If $m\ge n$, by Theorem \ref{square} and \ref{coincide},

   $$\begin{aligned}
    \left|\Gdet(AB) \right|
    &=\sqrt{\det(B^{T}A^{T}AB)}&\\
    &=\sqrt{\det(B^T)\cdot \det(A^{T}A)\cdot \det(B)}&\\
    &=\left|\Gdet(A)\cdot \det(B)\right|\\
    &=|\Gdet(A)||\Gdet(B)|.
    \end{aligned}
    $$

By Theorem \ref{signlemma} (3), $\sign(AB)=\sign(A)\sign(B)$.

Hence $\Gdet(AB)=\Gdet(A)\Gdet(B)$.
  \end{proof}
  
  If the definition of the case $m<n$ is $\Gdet(A):=\Gdet(A^T)$ as in Remark \ref{remark3}, then Theorem \ref{multiplication} will fail. For example, let $A=(1,0)$, $B=\begin{pmatrix}
  1 & 0\\
  0 & 2\\
  \end{pmatrix}$, then
  $$
  \Gdet (AB)=1\ne 2=\Gdet(A)\Gdet(B).
  $$
\begin{remark}
One may hope the dual version of Theorem \ref{multiplication} also holds: Let  $A\in\R^{m\times n}, B\in\R^{m\times m}$. Then $\Gdet(BA)  =\Gdet(B)\Gdet(A)$. Unfortunately, this is not true. For instance,
$$\Gdet\left[\mat{cc}{1&0\\ 0&2}\mat{c}{1\\ 0}\right]=\Gdet\mat{c}{1\\ 0}=1,$$
while
$$\Gdet\mat{cc}{1&0\\ 0&2}\Gdet\mat{c}{1\\ 0}=2.$$
\end{remark}

The Definition \ref{Gdet} also generalizes the Cauchy-Binet's formula in the case $A^TA$.
\begin{theorem}\label{C-B}
 Let $m\ge n$, $A\in \R^{m\times n}$. Then for all $n\leq k\leq m$, 
    $$
	\dbinom{k-n}{m-n}\det(A^TA)=\sum_{1\leq i_1<\ldots <i_k\leq m}{(\Gdet(A_{i_1,i_2,\ldots ,i_k}))^2},
	$$
where $A_{i_1,i_2,\ldots ,i_k}$ is the $k\times n$ matrix $\mat{c}{A^Te_{i_{1}},  \cdots, A^Te_{i_{k}}}^T$. 
\end{theorem}
\begin{proof}
If $k=n$, the theorem is part of Definition \ref{Gdet}. Suppose $n<k\leq m$, then
$$
\begin{aligned}
& \sum_{1\leq i_1<\ldots <i_k\leq m}{(\Gdet(A_{i_1,i_2,\ldots ,i_k}))^2}\\
=&\sum_{1\leq i_1<\ldots <i_k\leq m}{\sum_{1\leq j_1<\ldots <j_n\leq k}{(\Gdet(A_{i_{j_1},\ldots ,i_{j_n}}))^2}}\\
=&\sum_{1\leq i_1<\ldots <i_n\leq m}{\sum_{\{j_1,\ldots,j_{k-n}\}\ne \{i_1,\ldots,i_n\}}{(\Gdet(A_{i_1,\ldots ,i_n}))^2}}\\
=&\dbinom{k-n}{m-n}\sum_{1\leq i_1<\ldots <i_n\leq m}{(\Gdet(A_{i_1,\ldots ,i_n}))^2}\\
=&\dbinom{k-n}{m-n}\det(A^TA).
\end{aligned}
$$
\end{proof}

\begin{remark}
 From Theorem \ref{square}, Theorem \ref{multialf}, Theorem \ref{multiplication} and Theorem \ref{C-B}, we know that the generalized determinant given by Definition \ref{Gdet} keeps the most important properties (compatible with the usual determinant of square matrices, being an alternating multilinear function, keeping  multiplication formula and partly keeping the Cauchy-Binet's formula) of the usual determinant. 
\end{remark}

\section{Applications}

In this section, three more important applications of the usual determinant are generalized to rectangular matrices.


\subsection{The Generalized Cramer's Rule}

One of the most interesting applications of Definition \ref{Gdet} is that Cramer's Rule holds as usual. 

\begin{theorem}\label{Cramer}  Let $m\ge n, A\in\R^{m\times n}, b\in R(A)$, the  column space of $A$. Then the system of linear equations $Ax = b$ has exactly one solution if and only if $\Gdet(A)\not=0$. In this case, the solution $x=(x_i)_{i=1}^n$ can be represented by the generalized Cramer's Rule:
	$$
	x_i=\frac{\Gdet(A_i)}{\Gdet(A)},\ 1\le i\le n,
	$$
	where $A_i=(Ae_1, \cdots, Ae_{i-1}, b, Ae_{i+1}, \cdots, Ae_n)$.
\end{theorem}

\begin{proof} The first statement is true due to Corollary \ref{leftinverse}. To check the generalized Cramer's Rule, we make use of G. Strang's famous trick (cf.\cite{S}, Page 273-274). Set $B= (e_1, \cdots, e_{i-1}, x, e_{i+1}, \cdots, e_n)_{n\times n}$. Then
$$AB=(Ae_1, \cdots, Ae_{i-1}, b, Ae_{i+1}, \cdots, Ae_n)=A_i.$$
 
 It is easy to see that $\Gdet(B)=\det(B)=x_i$. Now, by the multiplication formula (Theorem \ref{multiplication}), we have $$\Gdet(A_i)=\Gdet(A)\Gdet(B) =\Gdet(A)x_i,$$ the generalized Cramer's Rule follows.

\end{proof}
Theorem \ref{Cramer} generalizes the Cramer's Rule in the most natural way in the following sense: If $\Gdet(A)=0$, then $Ax=b$ has no solution or infinitely many solutions, there is no generalized Cramer's Rule at all. It also shows the rationality of the case $m<n$ in Remark \ref{remark3}.
%
%
%

   \subsection{Subset presentation}

    From now on, we present two more applications of the generalized determinant given by Definition \ref{Gdet} in geometry. The first is the subset presentation in affine geometry. By abuse of language, we do not distinguish vectors and points in the affine space $\mathbb A^m$ and the vector space $\R^m$.

    \begin{theorem}\label{subspace} Let $V\subseteq\R^m$ be a proper linear subspace  with a basis $\alpha _1, \cdots, \alpha _n$. Then

    \[V=\{x\in\R^m | 
    \Gdet(\alpha _1, \cdots, \alpha _n, x)=0\}.
    \]

    \end{theorem}
   
    \begin{proof} According to Theorem \ref{multialf}, $\Gdet(\alpha _1, \cdots, \alpha _n, x)=0$ if and only if the involved $n+1$ vectors  $\alpha _1, \cdots, \alpha _n, x$ are linearly dependent, thus $x$ is linearly dependent with $\alpha _1, \cdots, \alpha _n$, since theses $n$ vectors form a basis of $V$. Therefore $x\in V$, as required. 
     \end{proof}
   
    We now consider the general linear variety in the affine space $\R^m$.  By {\it linear variety} we mean a subset in $\R^m$ of the form $V+\al_0$, where $\al_0\in\R^m$ is a fixed vector and $V$ a linear subspace of $\R^m$. 

    The following  is an immediate consequence of Theorem \ref{subspace}.
   
    \begin{corollary}\label{linearvariety} Given $\al_0\in\R^m$. Let $V\subseteq\R^m$ be a proper linear subspace  with a basis $\alpha _1, \cdots, \alpha _n$.  Then
    
    \[V+\al_0=\{x\in\R^m |
    \Gdet(\alpha _1, \cdots, \alpha _n, x-\al_0)=0\}.
    \]
    \end{corollary}

 
    \subsection{The Generalized Oriented Volume}
   
  It is well-known that the geometric meaning of the absolute value of the determinant of a matrix $A\in\R^{m\times m}$ is the volume of the parallel polyhedron generated by the $m$ column vectors of $A$, and the sign shows its orientation. However, the oriented volume of the parallel polyhedron generated by $n (< m)$ vectors in $\R^m$ does not make sense in the literature. In this subsection, we solve this problem by the sign system (Definition \ref{2d1-sign}) and Definition \ref{Gdet} of the generalized determinant for matrices.
  
  Now we give the exact definition of the {\it generalized volume} $GV(P)$ of an $n$-dimensional parallel polyhedron $P$ in $\R^{m}$. 
  
   Let $P\subseteq \R^{m}$ be the parallel polyhedron generated by $n$ linearly independent vectors $\alpha_1,\ldots,\alpha_n\in \R^{m}(m\ge n)$, that is $P=\{\sum\limits_{i=1}^n{x_i\alpha_i}|0\leq x_i\leq 1, i=1,\ldots,n\}$, $A=\left(\alpha_1,\ldots,\alpha_n\right)$.  Let $A=QR$ be the $QR$ decomposition. Define a linear map 
  $$
  \phi:\text{span}\{\alpha_1,\ldots,\alpha_n\}\rightarrow \R^{n}, \phi(q_i)=e_i,
  $$
  where $q_i$ is the $i$th column vector of $Q$. Note that $V(\phi(P))$, the volume of $\phi(P)(\subseteq \R^{n})$ is well-defined, and since $\phi$ is clearly an orthogonal transformation, we define $GV(P)$, the {\it generalized volume} of $P(\subseteq \R^{m})$, to be $V(\phi(P))$. In particular, if $m=n$, then $GV(P)=V(P)$.
  
    \begin{theorem}\label{volume1}
    Let $\alpha_1,\cdots ,\alpha_n\in\R^m $ be linearly independent. Denote by $P$ the parallel polyhedron generated by $\alpha _1,\cdots ,\alpha _n$.  Then
     $$GV(P) = |\Gdet (\alpha _1, \cdots, \alpha _n)|.$$
    \end{theorem}
    \begin{proof}
    Write $A=\left(\alpha_1,\ldots,\alpha_n\right)$. Let $A=QR$ be the $QR$ decomposition. Since $\phi(\alpha_i)=r_i$ is the $i$th column vector of $R$, $\phi(P)$ is the parallel polyhedron generated by $r_1,\ldots,r_n$. Hence $V(\phi(P))=\det(R)$. Then by Theorem \ref{coincide} and \ref{multiplication},
    $$
    |\Gdet(A)|=|\Gdet(Q)\Gdet(R)|=[\det(Q^TQ)]^{\frac{1}{2}}\det(R)=\det(R).
    $$
    
    Therefore, $GV(P)=V(\phi(P))=\det(R)=|\Gdet(A)|$.
    \end{proof}

\begin{remark}
From Corollary \ref{QR}, the sign of $A$ equals to the sign of $Q$, which shows the orientation of the column vectors of $Q$ orthogonally projecting on the $i_1i_2\ldots i_n$-subspace of $\R^{m}$, where $i_1,i_2,\ldots, i_n$ are the serial numbers of rows of the principle submatrix of $Q$, and the matrix of the orthogonal projection is $P=\sum\limits_{k=1}^{n}{e_k e_k^T}\in \R^{m\times m}$.
\end{remark}

 \begin{example}\label{length}
Let $\alpha\in\R^m$. Then $|\Gdet(\alpha)|$ is exactly the usual Euclidean length of $\alpha$. If $\al\ne0$, then the sign of $\Gdet(\alpha)$ is exactly the sign of first non-zero entry of $\al$. 
  \end{example} 
  
 \begin{example}\label{volume}
Let  $\alpha_1,\alpha_2$ be two vectors in $\R^3$, then $|\Gdet(\alpha_1,\alpha_2)|$ is the area of the parallelogram formed by $\alpha_1$ and $\alpha_2$ in $\R^3$. For instance, if
$\alpha_1=\mat{c}{3\\4\\2}$, $\alpha_2=\mat{c}{6\\8\\ 1}$,
then $\Gdet
\mat{cc}
{3&6\\ 
4&8\\
2&1}=-15$, namely $GV(P)=15$.

The negative sign of $-15$ shows the orientation of the projection of $\alpha_1,\alpha_2$ on the $xz$-plane.
\end{example} 

	\end{document}